\documentclass[12pt,reqno]{amsart}
\usepackage{amsfonts}
\usepackage{amsmath}
\usepackage{amssymb,latexsym}
\usepackage{enumerate}
\usepackage[bookmarksnumbered, colorlinks, plainpages]{hyperref}
\usepackage{amsbsy}
\usepackage{amscd}
\usepackage{epsfig}
\usepackage{graphicx}
\usepackage{epstopdf}
\usepackage{caption}
\usepackage{subcaption}

\newtheorem{theorem}{Theorem}[section]
\newtheorem{lemma}{Lemma}[section]
\newtheorem{proposition}{Proposition}[section]
\newtheorem{corollary}{Corollary}[section]
\newtheorem{definition}{Definition}[section]

\newtheorem{example}{Example}[section]
\newtheorem{remark}{Remark}[section]
\numberwithin{equation}{section}

\begin{document}
\date{{\scriptsize Received: , Accepted: .}}
\title[New Multivalued Contractions and the Fixed-Circle Problem]{New
Multivalued Contractions and the Fixed-Circle Problem}
\author[N. TA\c{S}]{Nihal TA\c{S}}
\address{Bal\i kesir University, Department of Mathematics, 10145 Bal\i
kesir, Turkey}
\email{nihaltas@balikesir.edu.tr}
\author[N. \"{O}ZG\"{U}R]{Nihal \"{O}ZG\"{U}R}
\address{Bal\i kesir University, Department of Mathematics, 10145 Bal\i
kesir, Turkey}
\email{nihal@balikesir.edu.tr}
\maketitle

\begin{abstract}
In this paper, we focus on the fixed-circle problem on metric spaces by means of the multivalued mappings. We introduce new multivalued contractions using Wardowski's
techniques and obtain new fixed-circle results related to multivalued contractions with some
applications to integral type contractions. We verify the validity of our
obtained results with illustrative examples.\newline
\end{abstract}




%

%

\section{\textbf{Introduction and Preliminaries}}

\label{intro} Fixed-point theory started with the Banach's contraction
principle given in \cite{Banach}:

Let $(X,d)$ be a complete metric space and $T:X\rightarrow X$ be a
self-mapping. If there exists a constant $L\in \lbrack 0,1)$ such that%
\begin{equation*}
d(Tx,Ty)\leq Ld(x,y)\text{,}
\end{equation*}%
for all $x,y\in X$ then $T$ has a unique fixed point \cite{Banach}. This
principle is one of the most efficient tools for many existence and
uniqueness problems in mathematics. Therefore, the Banach's contraction
principle has been generalized with different approaches. One of these
approaches is to generalize the used contractive condition (for example, see
\cite{Caristi, Chatterjea, Ciric1, Edelstein, Kannan, Nemytskii, Rhoades}). Another approach is to
generalize the used metric space (for example, see \cite{an, b-metric, 2-metric, M4, Mlaiki-Mb, G-metric, Sedghi-S-metric, Tas-parametric-S-metric}). Recently, the
fixed-circle problem has been studying as a new approach to generalize the
Banach's contraction principle (see \cite{Ozgur-malaysian, Ozgur-fixed-circle-S-metric, Ozgur-circle-tez, Ozgur-Aip,
Ozgur-new-chapter, Tas-parametric-Nb, Tas-Ozgur-Mlaiki, Tas-Ozgur-Mlaiki-2, Tas}).

On the other hand, fixed-point theory has many applications in mathematical
research areas and in the other science branches such as economy,
engineering etc. (for example, see \cite{Border, Ceng, Chen,
Fleiner, Ok}). For this reason, some contraction structures
has been studying for both single and multivalued mappings. In the
literature, there are some different studies for a multivalued mapping. In
1969, Nadler proved some fixed-point theorems for multivalued contraction
mappings using the notion of Hausdorff metric \cite{Nadler}. Later, in 1975,
Dube studied on common fixed-point theorems for multivalued mappings \cite%
{Dube}. In 1989, Mizoguchi and Takahashi obtained a multivalued version of
Caristi's fixed-point theorem \cite{Mizoguchi}. After then, in 1995, Daffer
and Kaneko studied on the paper \cite{Mizoguchi} and obtained some new
results \cite{Daffer}. In 2009, \'{C}iri\'{c} generalized some known results
related to multivalued nonlinear contraction mappings \cite{Ciric}. In 2016,
Altun et al. defined a new notion of a multivalued nonlinear $F$-contraction
and obtained some generalized fixed-point results \cite{Altun 2016}. For
further studies one can consult \cite{Aydi, Kamran, Kodama,
Lisica}.

Now we recall some useful definitions and properties related to multivalued
mappings.

Let $(X,d)$ be a metric space. Let $P(X)$ denote the family of all nonempty
subsets of $X$, $CB(X)$ denote the family of all nonempty, closed and
bounded subsets of $X$ and $K(X)$ denote the family of all nonempty compact
subsets of $X$. It is obvious that $K(X)\subseteq CB(X)$. The function $%
H:CB(X)\times CB(X)\rightarrow
\mathbb{R}
$ is defined by, for every $A,B\in CB(X)$,%
\begin{equation*}
H(A,B)=\max \left\{ \underset{x\in A}{\sup }D(x,B),\underset{y\in B}{\sup }%
D(y,A)\right\}
\end{equation*}%
is a metric on $CB(X)$ which is called Hausdorff metric induced by $d$ \cite%
{Kelley}, where%
\begin{equation*}
D(x,B)=\inf \left\{ d(x,y).y\in B\right\} \text{.}
\end{equation*}

\begin{lemma}
\cite{Dube} \label{lem1} Let $A,B\in CB(X)$ then for any $a\in A$,%
\begin{equation*}
D(a,B)\leq H(A,B)\text{.}
\end{equation*}
\end{lemma}

\begin{definition}
\cite{Nadler} \label{def5} Let $(X,d)$ be a metric space and $T:X\rightarrow
CB(X)$ be a mapping. Then $T$ is called a multivalued contraction if there
exists $L\in \lbrack 0,1)$ such that%
\begin{equation*}
H(Tx,Ty)\leq Ld(x,y)\text{,}
\end{equation*}%
for all $x,y\in X$.
\end{definition}

\begin{definition}
\cite{Nadler} \label{def6} A point $x\in X$ is called a fixed point of a
multivalued mapping $T:X\rightarrow CB(X)$ such that $x\in Tx$.
\end{definition}

Nadler proved that every multivalued contraction mapping has a fixed point
on a complete metric space \cite{Nadler}.

More recently, in \cite{Wardowski}, Wardowski defined the family of
functions as follows:

Let $\mathcal{F}$ be the family of all functions $F:(0,\infty )\rightarrow
\mathbb{R}
$ such that

$(F_{1})$ $F$ is strictly increasing, that is, for all $\alpha ,\beta \in
(0,\infty )$ such that $\alpha <\beta $, $F(\alpha )<F(\beta )$,

$(F_{2})$ For each sequence $\left\{ \alpha _{n}\right\} $ in $\left(
0,\infty \right) $ the following holds%
\begin{equation*}
\underset{n\rightarrow \infty }{\lim }\alpha _{n}=0\text{ if and only if }%
\underset{n\rightarrow \infty }{\lim }F(\alpha _{n})=-\infty \text{,}
\end{equation*}

$(F_{3})$ There exists $k\in (0,1)$ such that $\underset{\alpha \rightarrow
0^{+}}{\lim }\alpha ^{k}F(\alpha )=0$.

For example, if the functions $F_{i}:(0,\infty )\rightarrow
\mathbb{R}
$ $(i\in \{1,2,3,4\})$ defined as $F_{1}(x)=\ln (x)$ (resp. $F_{2}(x)=\ln
(x)+x$, $F_{3}(x)=-\frac{1}{\sqrt{x}}$ and $F_{4}(x)=\ln (x^{2}+x)$) satisfy
the conditions $(F_{1})$, $(F_{2})$ and $(F_{3})$, that is, $F_{i}\in
\mathcal{F}$ \cite{Wardowski}.

\begin{definition}
\cite{Wardowski} \label{def8} Let $(X,d)$ be a metric space and $%
T:X\rightarrow X$ be a mapping. Given $F\in \mathcal{F}$, $T$ is called an $%
F $-contraction if there exists $\tau >0$ such that%
\begin{equation*}
x,y\in X\text{, }d(Tx,Ty)>0\Longrightarrow \tau +F(d(Tx,Ty))\leq F(d(x,y))%
\text{.}
\end{equation*}
\end{definition}

Using together with Wardowski's technique and Nadler's technique, Altun et
al. introduced the notion of a multivalued $F$-contraction mapping and
proved some fixed-point results for these mappings \cite{Altun}.

\begin{definition}
\cite{Altun} \label{def7} Let $(X,d)$ be a metric space and $T:X\rightarrow
CB(X)$ be a mapping. Then $T$ is called a multivalued $F$-contraction if $%
F\in \mathcal{F}$ and there exists $\tau >0$ such that%
\begin{equation*}
x,y\in X\text{, }H(Tx,Ty)>0\Longrightarrow \tau +F(H(Tx,Ty))\leq F(d(x,y))%
\text{.}
\end{equation*}
\end{definition}

Notice that if we consider $F(\alpha )=\ln \alpha $, then every multivalued
contraction is also a multivalued $F$-contraction \cite{Altun}.

In 2014, Acar et al. introduced the notion of generalized multivalued $F$%
-contraction mappings using the number $M(x,y)$ defined as%
\begin{equation*}
M(x,y)=\left\{ d(x,y),D(x,Tx),D(y,Ty),\frac{1}{2}\left[ D(x,Ty)+D(y,Tx)%
\right] \right\}
\end{equation*}%
and generalized some multivalued fixed-point theorems \cite{Acar}.

\begin{definition}
\cite{Acar} \label{def9} Let $(X,d)$ be a metric space and $T:X\rightarrow
CB(X)$ be a mapping. Then $T$ is called a multivalued $F$-contraction if $%
F\in \mathcal{F}$ and there exists $\tau >0$ such that%
\begin{equation*}
x,y\in X\text{, }H(Tx,Ty)>0\Longrightarrow \tau +F(H(Tx,Ty))\leq F(M(x,y))%
\text{.}
\end{equation*}
\end{definition}

Using this number $M(x,y)$ and the integral type idea, Ojha and Mishra
studied some fixed-point results for multivalued mappings \cite{Ojha}. By
this approach and Wardowski's technique, Sekman et al. introduced a new
concept of the generalized multivalued integral type mapping under $F$%
-contraction with some important properties \cite{Sekman}.

\begin{definition}
\cite{Sekman} \label{def10} Let $(X,d)$ be a metric space and $%
T:X\rightarrow K(X)$ be a mapping. Then $T$ is called an $F$-contraction of
generalized multivalued integral type mapping if there exists $\tau >0$ such
that, for all $x,y\in X$%
\begin{equation*}
H(Tx,Ty)>0\Longrightarrow \tau +F\left( \int\limits_{0}^{H(Tx,Ty)}\varphi
(t)dt\right) \leq F\left( \int\limits_{0}^{d(x,y)}\varphi (t)dt\right) \text{%
,}
\end{equation*}%
where $\varphi :[0,\infty )\rightarrow \lbrack 0,\infty )$ is a
Lebesque-integrable mapping which is summable on each compact subset of $%
[0,\infty )$, non-negative and such that for each $\varepsilon >0$,%
\begin{equation*}
\int\limits_{0}^{\varepsilon }\varphi (t)dt>0\text{.}
\end{equation*}
\end{definition}

Very recently, Wardowski's technique has been used in the fixed-circle
problem for single valued mappings (see \cite{Tas-Ozgur-Mlaiki, Tas-Ozgur-Mlaiki-2}). At first, we recall the notion of a fixed circle on
metric spaces.

Let $(X,d)$ be a metric space. Then a circle and a disc are defined on a
metric space as follows, respectively:%
\begin{equation*}
C_{x_{0},r}=\left\{ x\in X:d(x,x_{0})=r\right\}
\end{equation*}%
and%
\begin{equation*}
D_{x_{0},r}=\left\{ x\in X:d(x,x_{0})\leq r\right\} \text{.}
\end{equation*}

\begin{definition}
\cite{Ozgur-malaysian} \label{def11} Let $(X,d)$ be a metric space, $%
C_{x_{0},r}$ be a circle and $T:X\rightarrow X$ be a self-mapping. If $Tx=x$
for every $x\in C_{x_{0},r}$ then the circle $C_{x_{0},r}$ is called as the
fixed circle of $T$.
\end{definition}

In the following definition, it was introduced a new type of a contractive
condition related to fixed circle.

\begin{definition}
\cite{Tas-Ozgur-Mlaiki} \label{def12} Let $(X,d)$ be a metric space. A
self-mapping $T$ on $X$ is said to be an $F_{c}$-contraction on $X$ if there
exist $F\in \mathcal{F}$, $t>0$ and $x_{0}\in X$ such that for all $x\in X$
the following holds$:$%
\begin{equation*}
d(x,Tx)>0\Rightarrow t+F(d(x,Tx))\leq F(d(x_{0},x))\text{.}
\end{equation*}
\end{definition}

Using this new contraction, it was obtained the following fixed-circle
(fixed-disc) result.

\begin{theorem}
\cite{Tas-Ozgur-Mlaiki} \label{thm5} Let $(X,d)$ be a metric space, $T$ be
an $F_{c}$-contractive self-mapping with $x_{0}\in $ $X$ and $r=\inf \left\{
d(x,Tx):x\neq Tx\right\} $. Then $C_{x_{0},r}$ is a fixed circle of $T$.
Especially, $T$ fixes every circle $C_{x_{0},\rho }$ where $\rho <r$.
\end{theorem}

Also the above contraction was generalized as follows:

\begin{definition}
\cite{Tas-Ozgur-Mlaiki-2} \label{def13} Let $(X,d)$ be a metric space and $T$
be a self-mapping on $X$. If there exist $F\in \mathbb{F}$, $t>0$ and $%
x_{0}\in X$ such that for all $x\in X$ the following holds$:$%
\begin{equation*}
d(x,Tx)>0\Longrightarrow t+F(d(x,Tx))\leq F(m(x,x_{0}))\text{,}
\end{equation*}%
where%
\begin{equation*}
m(x,y)=\max \left\{ d(x,y),d(x,Tx),d(y,Ty),\frac{1}{2}\left[ d(x,Ty)+d(y,Tx)%
\right] \right\} \text{,}
\end{equation*}%
then the self-mapping $T$ is called a \'{C}iri\'{c} type $F_{c}$-contraction
on $X$.
\end{definition}

Notice that the number $m(x,y)$ coincides with the number $M(x,y)$ for
single valued mappings.

\begin{theorem}
\cite{Tas-Ozgur-Mlaiki-2} \label{thm6} Let $(X,d)$ be a metric space, $T$ be
a \'{C}iri\'{c} type $F_{c}$-contraction with $x_{0}\in X$ and $r=\inf
\left\{ d(x,Tx):x\neq Tx\right\} $. If $d(x_{0},Tx)=r$ for all $x\in
C_{x_{0},r}$ then $C_{x_{0},r}$ is a fixed circle of $T$. Especially, $T$
fixes every circle $C_{x_{0},\rho }$ with $\rho <r$.
\end{theorem}

By the above motivations, we obtain some new fixed-circle results for
multivalued mappings using different multivalued contractions on metric
spaces. For this purpose, we define the notions of a multivalued $F_{C}$%
-contraction, a multivalued integral type $F_{C}$-contraction, a multivalued
\'{C}iri\'{c} type $F_{C}$-contraction and a multivalued integral \'{C}iri%
\'{c} type $F_{C}$-contraction. Using these new multivalued contractions, we
obtain some different approaches on the fixed-circle problem and prove new
fixed-circle theorems. In addition, we give some illustrative examples
related to the obtained results.

\section{\textbf{Main Results}}

\label{sec:1} In this section, we prove new fixed-circle results on metric
spaces. To do this, we define new types of multivalued contractive mappings
such as a multivalued $F_{C}$-contraction and a multivalued \'{C}iri\'{c}
type $F_{C}$-contraction with applications to integral type contractions.

\subsection{Multivalued $F_{C}$-Contractions and the Fixed-Circle Problem}

At first, we give the notion of a fixed circle for multivalued mappings as
follows:

\begin{definition}
\label{def0} Let $(X,d)$ be a metric space and $C_{x_{0},r}$ be a circle on $%
X$. A circle $C_{x_{0},r}$ is said to be a fixed circle of a multivalued
mapping $T:X\rightarrow CB(X)$ such that $x\in Tx$ for each $x\in
C_{x_{0},r} $.
\end{definition}

Now we define the following new multivalued contraction.

\begin{definition}
\label{def1} Let $(X,d)$ be a metric space and $T:X\rightarrow CB(X)$ be a
mapping. Then $T$ is said to be a multivalued $F_{C}$-contraction if there
exist $F\in \mathcal{F}$, $\tau >0$ and $x_{0}\in X$ such that%
\begin{equation*}
x\in X\text{, }H(Tx,A_{x})>0\Longrightarrow \tau +F(H(Tx,A_{x}))\leq
F(d(x,x_{0}))\text{,}
\end{equation*}%
where $A_{x}=\left\{ x\right\} $.
\end{definition}

Notice that there is a relation between the notions of a multivalued $F_{C}$%
-contraction and an $F_{C}$-contraction. Indeed, let $(X,d)$ be a metric
space and $T:X\rightarrow X$ be a self-mapping. If $T$ is an $F_{C}$%
-contraction, then $T$ can be considered as a multivalued $F_{C}$%
-contraction as follows:

Assume that $Tx=y$ for $x\in X$. Then we get $\{x\}\subset X$ and $\{y\}\in
CB(X)$. So using the $F_{C}$-contractive property, we obtain%
\begin{equation*}
H(Tx,A_{x})=H(\{y\},\{x\})=d(x,Tx)
\end{equation*}%
and%
\begin{equation*}
F(H(Tx,A_{x}))=F(d(x,Tx))\leq F(d(x,x_{0}))-\tau \text{.}
\end{equation*}%
Consequently, $T$ is a multivalued $F_{C}$-contraction.

Using Definition \ref{def1}, we obtain the following propositions.

\begin{proposition}
\label{prop1} Let $(X,d)$ be a metric space, $x\in X$, $A_{x}=\left\{
x\right\}$ and $T:X\rightarrow CB(X)$ be a mapping. Then we have%
\begin{equation*}
H(Tx,A_{x})=0\text{ if and only if }x\in Tx\text{.}
\end{equation*}
\end{proposition}

\begin{proof}
Let $H(Tx,A_{x})=0$. Using the definition of the metric $H$, we get%
\begin{eqnarray*}
H(Tx,A_{x}) &=&0\Longrightarrow H(Tx,\{x\})=0 \\
&\Longrightarrow &\max \left\{ \underset{y\in Tx}{\sup }D(y,\{x\}),\underset{%
z\in A_{x}}{\sup }D(z,Tx)\right\} =0 \\
&\Longrightarrow &D(y,\{x\})=0\text{, }y\in Tx\text{ and }D(z,Tx)\text{, }%
z\in A_{x} \\
&\Longrightarrow &d(y,x)=0\text{ and }D(x,Tx)=0 \\
&\Longrightarrow &x\in Tx\text{.}
\end{eqnarray*}%
The converse statement of this relation is clear.
\end{proof}

\begin{proposition}
\label{prop2} Let $(X,d)$ be a metric space and $T:X\rightarrow CB(X)$ be a
mapping. If $T$ is a multivalued $F_{C}$-contraction with $x_{0}\in X$ then
we have $x_{0}\in Tx_{0}$.
\end{proposition}

\begin{proof}
Assume that $H(Tx_{0},A_{x_{0}})>0$. Then using the multivalued $F_{C}$%
-contractive property, we obtain%
\begin{equation*}
\tau +F(H(Tx_{0},A_{x_{0}}))\leq F(d(x_{0},x_{0}))\text{,}
\end{equation*}%
which is a contradiction since the function $F$ is defined on $(0,\infty )$.
Therefore $H(Tx_{0},A_{x_{0}})=0$. By Proposition \ref{prop1}, we have $%
x_{0}\in Tx_{0}$.
\end{proof}

We give the following fixed-circle theorem.

\begin{theorem}
\label{thm1} Let $(X,d)$ be a metric space, $T:X\rightarrow CB(X)$ be a
multivalued $F_{C}$-contraction with $x_{0}\in X$ and
\begin{equation}
r=\inf \left\{ H(Tx,A_{x}):H(Tx,A_{x})>0\right\} .  \label{definition of r}
\end{equation}%
Then $C_{x_{0},r}$ is a fixed circle of $T$. Also, $T$ fixes every circle $%
C_{x_{0},\rho }$ with $\rho <r$.
\end{theorem}

\begin{proof}
Let $x\in C_{x_{0},r}$. If $H(Tx,A_{x})>0$ then using the multivalued $F_{C}$%
-contractive property and the fact that $F$ is increasing, we obtain%
\begin{eqnarray*}
F(r) &\leq &F(H(Tx,A_{x}))\leq F(d(x,x_{0}))-\tau \\
&=&F(r)-\tau <F(r)\text{,}
\end{eqnarray*}%
which is a contradiction. Therefore $H(Tx,A_{x})=0$ and by Proposition \ref%
{prop1}, we have $x\in Tx$.

Now we prove that $T$ also fixes any circle $C_{x_{0},\rho }$ with $\rho <r$%
. Let $x\in C_{x_{0},\rho }$ and suppose that $H(Tx,A_{x})>0$. By the
multivalued $F_{C}$-contractive property, we have%
\begin{eqnarray*}
F(H(Tx,A_{x})) &\leq &F(d(x,x_{0}))-\tau \\
&=&F(\rho )-\tau <F(\rho )<F(r)\text{,}
\end{eqnarray*}%
which is a contradiction with the definition of $r$. Hence $H(Tx,A_{x})=0$
and by Proposition \ref{prop1}, we get $x\in Tx$.
\end{proof}

We obtain the following corollaries.

\begin{corollary}
\label{cor1} Let $(X,d)$ be a metric space, $T:X\rightarrow CB(X)$ be a
multivalued $F_{C}$-contraction with $x_{0}\in X$ and $r$ be defined as in $%
( $\ref{definition of r}$)$. Then $T$ fixes the disc $D_{x_{0},r}$ and hence
$T $ fixes the center of any fixed circle.
\end{corollary}

\begin{proof}
If we consider the second part of Theorem \ref{thm1} and Proposition \ref%
{prop2}, then the proof can be easily seen.
\end{proof}

\begin{corollary}
\label{cor2} Let $(X,d)$ be a metric space, $T:X\rightarrow K(X)$ be a
multivalued $F_{C}$-contraction with $x_{0}\in X$ and $r$ be defined as in $%
( $\ref{definition of r}$)$. Then $C_{x_{0},r}$ is a fixed circle of $T$.
Also, $T$ fixes every circle $C_{x_{0},\rho }$ with $\rho <r$.
\end{corollary}

\begin{proof}
Since $K(X)\subset CB(X)$, by the similar arguments used in the proof of
Theorem \ref{thm1}, the proof can be easily seen.
\end{proof}

Notice that the converse statement of Theorem \ref{thm1} is not true
everywhen, that is, a mapping $T:X\rightarrow CB(X)$ can fix a circle
although it is not a multivalued $F_{C}$-contraction as seen in the
following example.

\begin{example}
\label{exm1} Let $X=[-2,0]\cup \left\{ x_{n}=1+\frac{1}{n}:n\in
\mathbb{N}
\right\} $ and $d(x,y)=\left\vert x-y\right\vert $ for all $x,y\in X$. Let
us define the mapping $T:X\rightarrow CB(X)$ as%
\begin{equation*}
Tx=\left\{
\begin{array}{ccc}
\{x\} & \text{if} & \left\vert x+1\right\vert \leq 1 \\
\{-1,x_{1},x_{2},\ldots ,x_{n-1}\} & \text{if} & x=x_{n}%
\end{array}%
\right. \text{,}
\end{equation*}%
for all $x\in X$. It is easy to see that $T$ fixes the circle $C_{-1,1}$ and
the disc $D_{-1,1}$, but $T$ is not a multivalued $F_{C}$-contraction with $%
x_{0}=-1$. Indeed, we get%
\begin{equation*}
H(Tx,A_{x})=d(-1,x_{n})>0\text{,}
\end{equation*}%
for all $x=x_{n}$. So using the symmetry condition and the fact that $F$ is
increasing, we obtain%
\begin{equation*}
H(Tx,A_{x})=d(x_{n},-1)\leq d(x_{n},-1)\Longrightarrow F(d(x_{n},-1))\leq
F(d(x_{n},-1))\text{.}
\end{equation*}%
In this case, there is no number $\tau >0$ to verify the condition of
Definition \ref{def1}. Consequently, $T$ is not a multivalued $F_{C}$%
-contraction for the point $x_{0}=-1$.
\end{example}

Now we give the following illustrative example.

\begin{example}
\label{exm2} Let $%
\mathbb{C}
$ be the set of all complex numbers, $X=%
\mathbb{C}
$ and $d(x,y)=\left\vert x-y\right\vert $ for all $x,y\in
\mathbb{C}
$. Let us define the mapping $T:%
\mathbb{C}
\rightarrow CB(%
\mathbb{C}
)$ as%
\begin{equation*}
Tx=\left\{
\begin{array}{ccc}
\{x\} & \text{if} & \left\vert x\right\vert <4 \\
\{x+1,x+2,x+3\} & \text{if} & \left\vert x\right\vert \geq 4%
\end{array}%
\right. \text{,}
\end{equation*}%
for all $x\in
\mathbb{C}
$. Then $T$ is a multivalued $F_{C}$-contraction with $F=\ln x$, $\tau =\ln
\frac{4}{3}$ and $x_{0}=0$. Indeed, we get%
\begin{equation*}
H(Tx,A_{x})=3>0\text{,}
\end{equation*}%
for $x\in
\mathbb{C}
$ with $\left\vert x\right\vert \geq 4$. Then we have%
\begin{eqnarray*}
3 &\leq &\left\vert x\right\vert \Longrightarrow \ln 3\leq \ln \left\vert
x\right\vert \\
&\Longrightarrow &\ln \frac{4}{3}+\ln 3\leq \ln \left\vert x\right\vert \\
&\Longrightarrow &\tau +\ln (H(Tx,A_{x}))\leq \ln (d(x,0))\text{.}
\end{eqnarray*}%
Also we obtain%
\begin{equation*}
r=\min \left\{ H(Tx,A_{x}):H(Tx,A_{x})>0\right\} =3\text{.}
\end{equation*}%
Consequently, $T$ fixes the circle $C_{0,3}$ and the disc $D_{0,3}$.
\end{example}

We note that the number of the fixed circles of the mapping $T$ defined in
Example \ref{exm2} is infinite but Theorem \ref{thm1} determine some of them.

\begin{definition}
\label{def2} Let $(X,d)$ be a metric space and $T:X\rightarrow CB(X)$ be a
mapping. Then $T$ is said to be a multivalued integral type $F_{C}$%
-contraction if there exist $F\in \mathcal{F}$, $\tau >0$ and $x_{0}\in X$
such that%
\begin{eqnarray*}
x &\in &X\text{, }H(Tx,A_{x})>0 \\
&\Longrightarrow &\tau +F\left( \int\limits_{0}^{H(Tx,A_{x})}\varphi
(t)dt\right) \leq F\left( \int\limits_{0}^{d(x,x_{0})}\varphi (t)dt\right)
\text{,}
\end{eqnarray*}%
where $A_{x}=\{x\}$ and $\varphi :[0,\infty )\rightarrow \lbrack 0,\infty )$
is a Lebesque-integrable mapping which is summable on each compact subset of
$[0,\infty )$, non-negative and such that for each $\varepsilon >0$,%
\begin{equation*}
\int\limits_{0}^{\varepsilon }\varphi (t)dt>0\text{.}
\end{equation*}
\end{definition}

\begin{remark}
\label{rem1} If we set the function $\varphi :[0,\infty )\rightarrow \lbrack
0,\infty )$ in Definition \ref{def2} as $\varphi (t)=1$ for all $t\in
\lbrack 0,\infty )$ then we get the definition of a multivalued $F_{C}$%
-contraction.
\end{remark}

Using the notion of a multivalued integral type $F_{C}$-contraction, we
obtain the following proposition and a new fixed-circle theorem.

\begin{proposition}
\label{prop3} Let $(X,d)$ be a metric space and $T:X\rightarrow CB(X)$ be a
mapping. If $T$ is multivalued integral type $F_{C}$-contraction with $%
x_{0}\in X$ then we have $x_{0}\in Tx_{0}$.
\end{proposition}

\begin{proof}
Suppose that $H(Tx_{0},A_{x_{0}})>0$. Then using the multivalued integral
type $F_{C}$-contractive property, we have%
\begin{equation*}
\tau +F\left( \int\limits_{0}^{H(Tx_{0},A_{x_{0}})}\varphi (t)dt\right) \leq
F\left( \int\limits_{0}^{d(x_{0},x_{0})}\varphi (t)dt\right) \text{,}
\end{equation*}%
a contradiction since the function $F$ is defined on $(0,\infty )$.
Therefore $H(Tx_{0},A_{x_{0}})=0$ and so by Proposition \ref{prop1}, we have
$x_{0}\in Tx_{0}$.
\end{proof}

\begin{theorem}
\label{thm2} Let $(X,d)$ be a metric space, $T:X\rightarrow CB(X)$ be a
multivalued integral type $F_{C}$-contraction with $x_{0}\in X$ and $r$ be
defined as in $($\ref{definition of r}$)$. Then $C_{x_{0},r}$ is a fixed
circle of $T$.
\end{theorem}

\begin{proof}
Let $x\in C_{x_{0},r}$. If $H(Tx,A_{x})>0$ then using the multivalued
integral type $F_{C}$-contractive property and the fact that $F$ is
increasing, we obtain%
\begin{eqnarray*}
F\left( \int\limits_{0}^{r}\varphi (t)dt\right) &\leq &F\left(
\int\limits_{0}^{H(Tx,A_{x})}\varphi (t)dt\right) \leq F\left(
\int\limits_{0}^{d(x,x_{0})}\varphi (t)dt\right) -\tau \\
&=&F\left( \int\limits_{0}^{r}\varphi (t)dt\right) -\tau <F\left(
\int\limits_{0}^{r}\varphi (t)dt\right) \text{,}
\end{eqnarray*}%
which is a contradiction. Therefore $H(Tx,A_{x})=0$ and by Proposition \ref%
{prop1}, we have $x\in Tx$.
\end{proof}

We give the following corollary.

\begin{corollary}
\label{cor3} Let $(X,d)$ be a metric space, $T:X\rightarrow CB(X)$ be a
multivalued integral type $F_{C}$-contraction with $x_{0}\in X$ and $r$ be
defined as in $($\ref{definition of r}$)$. Then $T$ fixes the disc $%
D_{x_{0},r}$ and so $T$ fixes the center of any fixed circle.
\end{corollary}

\begin{proof}
By the similar arguments used in the proofs of the second parts of Theorem %
\ref{thm1} and Theorem \ref{thm2}, the proof can be easily proved.
\end{proof}

\begin{remark}
\label{rem2} If we set the function $\varphi :[0,\infty )\rightarrow \lbrack
0,\infty )$ in Proposition \ref{prop3} $($resp. Theorem \ref{thm2} and
Corollary \ref{cor3}$)$ as $\varphi (t)=1$ for all $t\in \lbrack 0,\infty )$
then we get Proposition \ref{prop2} $($resp. Theorem \ref{thm1} and
Corollary \ref{cor1}$)$.
\end{remark}

\subsection{Multivalued \'{C}iri\'{c} Type $F_{C}$-Contractions and the
Fixed-Circle Problem}

We define the following new contraction called as a multivalued \'{C}iri\'{c}
type $F_{C}$-contraction.

\begin{definition}
\label{def3} Let $(X,d)$ be a metric space and $T:X\rightarrow CB(X)$ be a
mapping. Then $T$ is said to be a multivalued \'{C}iri\'{c} type $F_{C}$%
-contraction if there exist $F\in \mathcal{F}$, $\tau >0$ and $x_{0}\in X$
such that%
\begin{equation*}
x\in X\text{, }H(Tx,A_{x})>0\Longrightarrow \tau +F(H(Tx,A_{x}))\leq
F(M(x,x_{0}))\text{,}
\end{equation*}%
where $A_{x}=\left\{ x\right\} $ and%
\begin{equation*}
M(x,y)=\left\{ d(x,y),D(x,Tx),D(y,Ty),\frac{1}{2}\left[ D(x,Ty)+D(y,Tx)%
\right] \right\} \text{.}
\end{equation*}
\end{definition}

Let $(X,d)$ be a metric space and $T:X\rightarrow X$ be a self-mapping. If $%
T $ is a \'{C}iri\'{c} type $F_{C}$-contraction, then $T$ can be considered
as a multivalued \'{C}iri\'{c} type $F_{C}$-contraction. Indeed, assume that
$Tx=y$ for $x\in X$. Then we get $\{x\}\subset X$ and $\{y\}\in CB(X)$. So
using the \'{C}iri\'{c} type $F_{C}$-contractive property, we have%
\begin{equation*}
H(Tx,A_{x})=H(\{y\},\{x\})=d(x,Tx)
\end{equation*}%
and%
\begin{equation*}
F(H(Tx,A_{x}))=F(d(x,Tx))\leq F(m(x,x_{0}))-\tau =F(M(x,x_{0}))-\tau \text{,}
\end{equation*}%
where%
\begin{equation*}
m(x,y)=\max \left\{ d(x,y),d(x,Tx),d(y,Ty),\frac{1}{2}\left[ d(x,Ty)+d(y,Tx)%
\right] \right\} \text{.}
\end{equation*}

Using this new notion, we get the following proposition.

\begin{proposition}
\label{prop4} Let $(X,d)$ be a metric space and $T:X\rightarrow CB(X)$ be a
mapping. If $T$ is multivalued \'{C}iri\'{c} type $F_{C}$-contraction with $%
x_{0}\in X$ then we have $x_{0}\in Tx_{0}$.
\end{proposition}

\begin{proof}
Assume that $H(Tx_{0},A_{x_{0}})>0$. Then using the multivalued \'{C}iri\'{c}
type $F_{C}$-contractive property and Lemma \ref{lem1}, we get%
\begin{eqnarray*}
\tau +F(H(Tx_{0},A_{x_{0}})) &\leq &F(M(x_{0},x_{0})) \\
&=&F\left( \max \left\{
\begin{array}{c}
d(x_{0},x_{0}),D(x_{0},Tx_{0}),D(x_{0},Tx_{0}), \\
\frac{1}{2}\left[ D(x_{0},Tx_{0})+D(x_{0},Tx_{0})\right]%
\end{array}%
\right\} \right) \\
&=&F(D(x_{0},Tx_{0}))\text{,}
\end{eqnarray*}%
which is a contradiction. Hence we obtain $H(Tx_{0},A_{x_{0}})=0$ and so by
Proposition \ref{prop1}, we have $x_{0}\in Tx_{0}$.
\end{proof}

Now we prove another fixed-circle theorem.

\begin{theorem}
\label{thm3} Let $(X,d)$ be a metric space, $T:X\rightarrow CB(X)$ be a
multivalued \'{C}iri\'{c} type $F_{C}$-contraction with $x_{0}\in X$ and $r$
be defined as in $($\ref{definition of r}$)$. If $D(Tx,x_{0})=r$ then $%
C_{x_{0},r}$ is a fixed circle of $T$.
\end{theorem}

\begin{proof}
Let $x\in C_{x_{0},r}$. If $H(Tx,A_{x})>0$ then using the multivalued \'{C}%
iri\'{c} type $F_{C}$-contractive property, Lemma \ref{lem1} and the fact
that $F$ is increasing, we obtain%
\begin{eqnarray*}
F(H(Tx,A_{x})) &\leq &F(M(x,x_{0}))-\tau \\
&=&F\left( \max \left\{
\begin{array}{c}
d(x,x_{0}),D(x,Tx),D(x_{0},Tx_{0}), \\
\frac{1}{2}\left[ D(x,Tx_{0})+D(x_{0},Tx)\right]%
\end{array}%
\right\} \right) -\tau \\
&\leq &F\left( \max \left\{ r,H(Tx,A_{x}),0,\frac{1}{2}\left[ r+r\right]
\right\} \right) -\tau \\
&=&F(\max \{r,H(Tx,A_{x})\})-\tau \\
&=&F(H(Tx,A_{x}))-\tau <F(H(Tx,A_{x}))\text{,}
\end{eqnarray*}%
which is a contradiction. Therefore $H(Tx,A_{x})=0$ and we have $x\in Tx$ by
Proposition \ref{prop1}.
\end{proof}

\begin{corollary}
\label{cor4} Let $(X,d)$ be a metric space, $T:X\rightarrow CB(X)$ be a
multivalued \'{C}iri\'{c} type $F_{C}$-contraction with $x_{0}\in X$ and $r$
be defined as in $($\ref{definition of r}$)$. If $D(Tx,x_{0})=r$ then $T$
fixes the disc $D_{x_{0},r}$ \ and so $T$ fixes the center of any fixed
circle.
\end{corollary}

\begin{proof}
By the similar arguments used in the proofs of the second parts of Theorem %
\ref{thm1} and Theorem \ref{thm3}, the proof can be easily seen.
\end{proof}

\begin{corollary}
\label{cor5} Let $(X,d)$ be a metric space, $T:X\rightarrow K(X)$ be a
multivalued \'{C}iri\'{c} type $F_{C}$-contraction with $x_{0}\in X$ and $r$
be defined as in $($\ref{definition of r}$)$. If $D(Tx,x_{0})=r$ then $%
C_{x_{0},r}$ is a fixed circle of $T$.
\end{corollary}

\begin{proof}
Since $K(X)\subset CB(X)$, by the similar arguments used in the proof of
Theorem \ref{thm3}, it can be easily seen.
\end{proof}

Notice that the converse statements of Theorem \ref{thm3} and Corollary \ref%
{cor4} are not always true. Indeed, let us consider the metric space $(X,d)$
and the mapping $T$ defined as in Example \ref{exm1}. Then $T$ fixes the
circle $C_{-1,1}$ and the disc $D_{-1,1}$, but $T$ is not a multivalued \'{C}%
iri\'{c} type $F_{C}$-contraction for the point $x_{0}=-1$. For all $x=x_{n}$%
, we get%
\begin{eqnarray*}
H(Tx,A_{x}) &=&H(\{1,x_{1},x_{2},\ldots ,x_{n-1}\},\{x_{n}\}) \\
&=&\max \left\{ \underset{a\in Tx}{\sup }D(a,A_{x}),\underset{b\in A_{x}}{%
\sup }D(b,Tx)\right\} \\
&=&d(-1,x_{n})>0
\end{eqnarray*}%
and using the symmetry condition we have%
\begin{eqnarray*}
M(x,-1) &=&M(x_{n},-1) \\
&=&\max \left\{
\begin{array}{c}
d(x_{n},-1),D(x_{n},Tx_{n}),D(-1,T(-1)), \\
\frac{1}{2}\left[ D(x_{n},T(-1))+D(-1,Tx_{n})\right]%
\end{array}%
\right\} \\
&=&d(x_{n},-1)\text{.}
\end{eqnarray*}%
So using the fact that $F$ in increasing, we have%
\begin{equation*}
F(H(Tx,A_{x}))\leq F(M(x,-1))\text{.}
\end{equation*}%
In this case, there is no number $\tau >0$ to satisfy the condition of
Definition \ref{def3}. Hence $T$ is not a multivalued \'{C}iri\'{c} type $%
F_{C}$-contraction for the point $x_{0}=-1$.

Now we give the following example.

\begin{example}
\label{exm3} Let us consider the metric space $(X,d)$ and the mapping $T$
defined as in Example \ref{exm2}. Then $T$ is a multivalued \'{C}iri\'{c}
type $F_{C}$-contraction with $F=\ln x$, $\tau =\ln \frac{4}{3}$ and $%
x_{0}=0 $ since%
\begin{equation*}
\tau +\ln (H(Tx,A_{x}))\leq \ln (d(x,0))\leq \ln (M(x,0))\text{.}
\end{equation*}%
Consequently, $T$ fixes the circle $C_{0,3}$ and the disc $D_{0,3}$.
\end{example}

We define the following integral type contraction.

\begin{definition}
\label{def4} Let $(X,d)$ be a metric space and $T:X\rightarrow CB(X)$ be a
mapping. Then $T$ is said to be a multivalued integral \'{C}iri\'{c} type $%
F_{C}$-contraction if there exist $F\in \mathcal{F}$, $\tau >0$ and $%
x_{0}\in X$ such that%
\begin{eqnarray*}
x &\in &X\text{, }H(Tx,A_{x})>0 \\
&\Longrightarrow &\tau +F\left( \int\limits_{0}^{H(Tx,A_{x})}\varphi
(t)dt\right) \leq F\left( \int\limits_{0}^{M(x,x_{0})}\varphi (t)dt\right)
\text{,}
\end{eqnarray*}%
where $A_{x}=\left\{ x\right\} $ and $\varphi :[0,\infty )\rightarrow
\lbrack 0,\infty )$ is a Lebesque-integrable mapping which is summable on
each compact subset of $[0,\infty )$, non-negative and such that for each $%
\varepsilon >0$,%
\begin{equation*}
\int\limits_{0}^{\varepsilon }\varphi (t)dt>0\text{.}
\end{equation*}
\end{definition}

\begin{remark}
\label{rem3} If we set the function $\varphi :[0,\infty )\rightarrow \lbrack
0,\infty )$ in Definition \ref{def4} as $\varphi (t)=1$ for all $t\in
\lbrack 0,\infty )$ then we get the definition of a multivalued \'{C}iri\'{c}
type $F_{C}$-contraction.
\end{remark}

Using the notion of a multivalued integral \'{C}iri\'{c} type $F_{C}$%
-contraction, we get the following proposition and theorem.

\begin{proposition}
\label{prop5} Let $(X,d)$ be a metric space and $T:X\rightarrow CB(X)$ be a
mapping. If $T$ is multivalued integral \'{C}iri\'{c} type $F_{C}$%
-contraction with $x_{0}\in X$ then we have $x_{0}\in Tx_{0}$.
\end{proposition}

\begin{proof}
The proof can be easily seen by the similar arguments used in the proof of
Proposition \ref{prop3}.
\end{proof}

\begin{theorem}
\label{thm4} Let $(X,d)$ be a metric space, $T:X\rightarrow CB(X)$ be a
multivalued integral \'{C}iri\'{c} type $F_{C}$-contraction with $x_{0}\in X$
and $r$ be defined as in $($\ref{definition of r}$)$. If $D(Tx,x_{0})=r$
then $C_{x_{0},r}$ is a fixed circle of $T$.
\end{theorem}

\begin{proof}
From the similar arguments used in the proof of Theorem \ref{thm2}, the
proof can be easily checked.
\end{proof}

\begin{corollary}
\label{cor6} Let $(X,d)$ be a metric space, $T:X\rightarrow CB(X)$ be a
multivalued integral \'{C}iri\'{c} type $F_{C}$-contraction with $x_{0}\in X$
and $r$ be defined as in $($\ref{definition of r}$)$. If $D(Tx,x_{0})=r$
then $T$ fixes the disc $D_{x_{0},r}$ \ and so $T$ fixes the center of any
fixed circle.
\end{corollary}

\begin{proof}
The proof is clear by the similar arguments used in the proof of Corollary %
\ref{cor3}.
\end{proof}

\begin{remark}
If we set the function $\varphi :[0,\infty )\rightarrow \lbrack 0,\infty )$
in Proposition \ref{prop5} $($resp. Theorem \ref{thm4} and Corollary \ref%
{cor6}$)$ as $\varphi (t)=1$ for all $t\in \lbrack 0,\infty )$ then we get
Proposition \ref{prop4} $($resp. Theorem \ref{thm3} and Corollary \ref{cor4}$%
)$.
\end{remark}

\textbf{Acknowledgement.} This paper was supported by Bal\i kesir University
Research Grant no: 2018/021.

\end{document}